\documentclass[11pt]{amsart}
\usepackage{amsmath,amssymb,amsthm}

\usepackage[T1]{fontenc}
\usepackage[abs]{overpic}

\usepackage{enumerate}
\usepackage{tikz-cd}
\usepackage{hyperref}
\usepackage{amsthm,amsfonts,amsmath,amscd,amssymb}
\usepackage{xcolor}
\usepackage{mathrsfs}
\usepackage{graphics,graphicx}
\usepackage{caption}
\let\oldmarginpar\marginpar
\renewcommand\marginpar[1]{\-\oldmarginpar[\raggedleft\footnotesize #1]%
{\raggedright\footnotesize #1}}
\theoremstyle{plain}
\newtheorem{thm}{Theorem}[section]
\newtheorem{lemma}[thm]{Lemma}

\theoremstyle{definition}

\theoremstyle{remark}

\numberwithin{equation}{section}

\newcommand{\N}{\mathbb{N}}

\newcommand{\R}{\mathbb{R}}
\newcommand{\C}{\mathbb{C}}

\newcommand{\la}{\lambda}

\renewcommand{\a}{\alpha}

\newcommand{\sse}{\subseteq}

\newcommand{\Aut}{\operatorname{Aut}}

\newcommand{\std}{\text{st}}

\begin{document}
\begin{abstract}
We show that there exists a transverse link in the standard contact structures on the 3--sphere such that all contact 3--manifolds are contact branched covers over this transverse link. 

Figure 1 in this version of the paper differs from the published version, but the paper is otherwise unchanged. 
\end{abstract}

\title{Transverse universal links}

\author{Roger Casals}
\address{Massachusetts Institute of Technology, Department of Mathematics, 77 Massachusetts Avenue Cambridge, MA 02139, USA}
\email{casals@mit.edu}
\author{John B.~Etnyre}
\address{School of Mathematics, Georgia Institute
of Technology, 686 Cherry Street,  Atlanta, GA 30332-0160, USA}

\email{etnyre@math.gatech.edu}

%\address{School of Mathematics, Georgia Institute of Technology}
%\email{etnyre@math.gatech.edu}

\maketitle
% \vspace{-1.5cm}
\section{Introduction}\label{sec:intro}
In 1982, William Thurston showed that there exists a six component link in the 3--sphere such that any closed smooth oriented 3--manifold is the branched cover over the 3--sphere with branch locus this link \cite{ThurstonUniversalPrep}. He called such a link universal. Later H.~Hilden, M.~Lozano, and J.M.~Montesinos showed that there exists a universal knot in the 3--sphere \cite{HildenLozanoMontesinos83a, HildenLozanoMontesino85}. Since then there have been several articles showing certain knots and links are universal or not universal. For example the figure eight knot, Whitehead link, and Borromean rings are all universal \cite{HildenLozanoMontesinos85b, HildenLozanoMontesinos83b}. 

In 2002, Emmanuel Giroux showed that any contact 3--manifold is a 3--fold simple branched cover of the 3--sphere with the standard contact structure $(S^3,\xi_{std})$ and branch locus a transverse link \cite{Giroux02}. This is a contact strengthening of the Hilden-Montesinos Theorem \cite{Hilden74,Montesinos76} for smooth 3--manifolds. These contact constructions are useful for constructing open books for some contact 3--manifolds and embedding contact 3--manifolds in the standard contact structure on the 5--sphere \cite{EtnyreFurukawa17}.

\begin{figure}[h!]
\tiny
\begin{overpic}%[grid,tics=10] 
{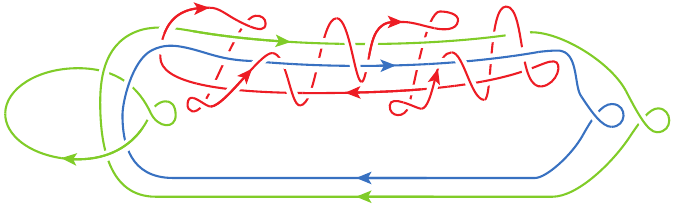}
%\put(-5,72){$r=$}
\end{overpic}
\caption{The universal transverse link in Theorem \ref{thm:univlink}.}
\label{tul}
\end{figure}

In view of the work of W.~Thurston and \cite{HildenLozanoMontesinos83a, HildenLozanoMontesino85}, it is a natural question in low-dimensional contact topology to ask whether there exists a universal transverse link in $(S^3,\xi_{std})$. That is, a transverse link $L\subseteq(S^3,\xi_{std})$ such that {\it any} contact 3--manifold is the contact branched cover of $(S^3,\xi_{std})$ along $L$. This question was first considered by M.~Casey in \cite{Casey13} where she showed that no transverse knot in the knot type of the figure eight can be (contact) universal and that many covers of the Whitehead link and Borromean rings yield only overtwisted contact structures, raising doubts as to whether or not there exist universal transverse links. The main result in the present work is that universal transverse links exist:

\begin{thm}\label{thm:univlink}
Any contact 3-manifold $(M,\xi)$ can be realized as a contact branched cover of $(S^3,\xi_{std})$ branched along the transverse link $L$ shown in Figure~\ref{tul}.
\end{thm}

The proof presented here directly adapts the argument for the existence of topologically universal links in \cite{HildenLozanoMontesinos83a} to the contact case. It is interesting to note that similar constructions of universal links that lead to universal knots cannot be easily adapted to the contact setting, leaving open the natural question 
%of whether there exists a universal transverse knot in $(S^3,\xi_{std})$. So we ask the natural question

\begin{itemize}
 \item[(a)] Does there exist a universal transverse knot $K\subseteq (S^3,\xi_{std})$ ?
\end{itemize}

The Figure-8 knot is a smooth universal knot, but it cannot be a contact universal knot after the results from M.~Casey's thesis \cite{Casey13}. This leads to the following problem:

\begin{itemize}
 \item[(b)] Find sufficient conditions for a smooth universal knot or link to admit transverse representatives which are contact universal.
\end{itemize}

In line with the first question, we would also like to know the answer to the following question:
\begin{itemize}
 \item[(c)] Let $(S^3,\xi_{ot})$ be an overtwisted contact structure. Does there exist a transverse knot $K\subseteq (S^3,\xi_{ot})$ such that the contact branched cover of $(S^3,\xi_{ot})$ along $K$ is contactomorphic to $(S^3,\xi_{std})$ ?
\end{itemize}

It is possible to produce branched covers of overtwisted structures on $S^3$ which are tight {\em cf.\ }\cite{Casey13}, yet it is apparent that such $K$ in the question --- should it exist --- cannot be a smooth unknot.

Finally, a natural continuation of Theorem~\ref{thm:univlink}, in line with \cite{LoiPiergallini01,PiergalliniZuddas05}, would be to answer the following question:

\begin{itemize}
 \item[(d)] Does there exist a symplectic surface $S\subseteq (D^4,\la_{std})$ such that any Weinstein 4-fold $(W,\la)$ is a branched cover of $(D^4,\la_{std})$ along $S$ ?
\end{itemize}

{\bf Acknowledgements:} We thank Eric Stenhede, Jesús Rodríguez-Viorato, and Sebastian Zapata-Rendón for pointing out an error in the original version of Figure 1. We thank the referee for helpful comments that improved the paper. The first author is supported by the NSF grants DMS-1608018 and, subsequently, DMS-1841913, and a BBVA Research Fellowship. He also thanks Ad\'an Medrano Mart\'in del Campo for useful conversations. The second author was partially supported by NSF grant DMS-1608684.

\section{Background}
In this section we review the necessary definitions on transverse knots and branched covers which we will use in our proof of Theorem \ref{thm:univlink}. We also show, in the last subsection, that many standard modifications of the branch locus for branched covering maps that are known in the topological setting also hold in the contact geometric setting as well. 

\subsection{Transverse knots}\label{transknots}
For more details on transverse knots the reader is referred to \cite{Etnyre05}. Here we briefly review the aspects that will be relevant for our main result. 

We will consider our knots in $(\R^3,\xi_{std})\subset (S^3,\xi_{std})$ where on $\R^3$ we have the standard contact structure
\[
\xi_{std}=\ker (dz-y\, dx),
\]
and $(x,y,z)$ are Cartesian coordinates on $\R^3$. A knot $K$ is {\em transverse} if $K$ is transverse to $\xi_{std}$ at every point of $K$. Since $\xi_{std}$ is co-oriented by the contact form and $\R^3$ is oriented, $K$ will have an orientation induced on it so that $K$ and $\xi_{std}$ intersect positively. 

We will study transverse knots via their {\it front projection}, that is, via the image of $K$ under the projection $\pi:\R^3\to \R^2: (x,y,z)\mapsto (x,z)$. Notice that being transverse to $\xi_{std}$ implies that the $y$--coordinate of $K$ satisfies
\[
y< \frac{dz}{dx},
\] 
where we think of the page as the $xz$-plane so that the positive $y$-axis points into the page. It is easy to see that any diagram for an oriented knot in the $xz$-plane determines a transverse knot up to isotopy through transverse knots as long as no portion of the knot is as shown in Figure~\ref{forbid}.
\begin{figure}[h]
\tiny
\begin{overpic}%[grid,tics=10] 
{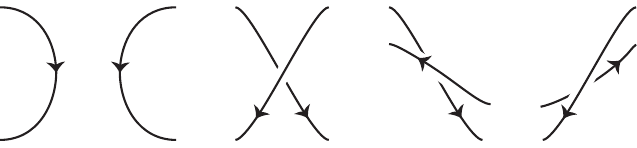}
%\put(-5,72){$r=$}
\end{overpic}
\caption{Forbidden portions of the front diagram of a transverse knot.}
\label{forbid}
\end{figure}  
Moreover, Type II and III Reidemeister moves are allowed in front diagrams as long as no portion of the move contains a forbidden diagram from Figure~\ref{forbid}. More specifically, there can be no vertical tangencies with the oriented tangent vector pointing down and there can be no crossings where the over strand has greater slope and is oriented right to left, while the under crossing is oriented left to right. 

In addition to the {\it front projections}, there is a different presentation of transverse knots that will be useful for us as well. Consider $S^3=\{(z_1,z_2)\in\C^2:|z_1|^2+|z_2|^2=1\}$ as the unit sphere in $\C^2$ with the standard contact structure $\xi_{std}=TS^3\cap i(TS^3)$ given as the set of complex tangencies to $S^3$. The link $H=\{z_1=0\}\cup \{z_2=0\}$ is a Hopf link with each component being an unknot with self-linking number $-1$. It is easy to check that the contact structure on the complement $S^3- H=T^2\times (0,1)$ is given by 
\[
\xi_{std}=\ker\left(\cos \left(\frac\pi2 t\right) \, d\phi + \sin \left(\frac\pi2 t\right) \, d\theta\right),
\] 
where $(\phi,\theta)$ are angular coordinates on $T^2=S^1\times S^1$ and $t$ is the coordinate on $(0,1)$. We can now consider transverse knots in $T^2\times (0,1)$ and their {\it front projection} will be obtained by projecting out the $t$-coordinate. Once again a transverse knots can be recovered from its front projection. The only real difference with the situation described above is that the tangent vector to the projection cannot have both $\phi$ and $\theta$ component negative. 

A closed transverse braid in $\R^3$ can always be assumed to lie in the thickened torus $T^2\times (0,1)$ and such braids are in one-to-one correspondence curves in $T^2\times (0,1)$ whose oriented tangents always have positive $\theta$-coordinate \cite{Bennequin83}. 

\subsection{Contact branched covers}
A {\em branched covering map} is a smooth map $p:M\to Y$ between smooth 3--manifolds such that there is a link $L\subseteq Y$, called the {\em branch locus}, such that $p$ restricted to $M-p^{-1}(L)$ is a covering map from $M-p^{-1}(L)$ to $Y-L$ and each component of $p^{-1}(L)$ has a neighborhood $N=S^1\times D^2$  such that $p(N)=S^1\times D^2$ and in these coordinates $p$ is of the form $(\phi,z)\mapsto (n\phi, z^m)$ for some integers $n$ and $m$ where $\phi$ is the angular coordinate on $S^1$ and $z$ is the complex coordinate on $D^2$ thought of as the unit disk in $\C$. We say the component has {\em order $m$} and we say $p$ is {\em ramified} at the component if $m>1$. An branched covering map of degree $k$ is called {\em simple} if the pre-image of any point has size either $k$ or $k-1$. 

Given a contact submanifold, it is quite easy to construct contact structures on branched covering spaces from a contact structure on the base space. 
\begin{thm}[Geiges 1997, \cite{Geiges97}; \"Ozt\"urk and Niederkr\"uger 2007, \cite{OzturkNiederkruger07}]
Let $p:M\to Y$ be a branched covering map with branch locus $L$. Given a contact structure $\xi$ on $Y$ with contact from $\alpha$ such that $L$ is transverse to $\xi$, then $p^*\alpha$ may be deformed by an arbitrarily small amount near $p^{-1}(L)$ to give a contact form defining a unique, up to contact isotopy, contact structure $\xi_L$ on $M$. 
\end{thm}
The starting point for our main result is Giroux's proof that every contact 3--manifold can be constructed as a branched cover over the standard contact structure on $S^3$.
\begin{thm}[Giroux 2002, \cite{Giroux02}]\label{giroux}
Given a contact structure $\xi$ on a 3--manifold $M$, there is a 3--fold simple branched covering map $p:M\to S^3$ such that $\xi$ is the contact structure induced on $M$ by $p$, the standard contact structure on $S^3$, and some transverse realization of the branch locus of $p$. 
\end{thm}

For each contact 3--manifold $(M,\xi)$, Theorem \ref{giroux} asserts the existence of a transverse link $L(M,\xi)\subseteq (S^3,\xi_{std})$, depending on $(M,\xi)$, such that $(M,\xi)$ is the contact branched cover of $(S^3,\xi_{std})$ along $L(M,\xi)$. Our main result Theorem \ref{thm:univlink} shows that the transverse link $L(M,\xi)$ can be chosen to be independent of $(M,\xi)$, thus encoding all the complexity of the contact structure on $(M,\xi)$ in the coverings of $S^3\setminus L$.

Let us briefly recall the argument for Theorem \ref{giroux}. E.~Giroux first shows that any contact 3--manifold is supported by an open book decomposition, whose binding can be assumed to be connected. Building on work of J.S.~Birman \cite{Birman79} and D.L.~Goldsmith \cite{Goldsmith75}, J.M.~Montesinos and H.R.~Morton showed in \cite{Montesinos-AmilibiaMorton91} that any open book decomposition on a smooth 3--manifold $M$ with connected binding is pulled back from the open book decomposition of $S^3$ with disk page by a 3--fold branched cover map $p:M\to S^3$ whose branch locus is a braid. Since a braid is naturally a transverse knot the branched cover can be taken to be a contact branched cover and it is easy to see that the induced contact structure on the covering space is supported by the original open book, thus concluding the statement.

We end this subsection with two simple observations. The first follows immediately from the definition of contact branched cover. 
\begin{lemma}\label{composition}
Suppose $p:(M',\xi')\to (M,\xi)$ is a contact branched covering map with branch set $B\subset M$. Let $L$ be a link in $M$ disjoint from $B$ and $L'=p^{-1}(L)$. If $p':(M'',\xi'')\to (M',\xi')$ is a contact branched covering map with branch set $L'$, then $p'\circ p:(M'',\xi'')\to (M,\xi)$ is a contact branched covering map with branch set $B\cup L$. \hfill\qed
\end{lemma}

The second fact is also standard in contact topology.

\begin{lemma}\label{firstbc}
The contact manifold obtained from any branched cover of $(S^3,\xi_{std})$ branched along a transverse unknot with self-linking number $-1$ is contactomorphic to $(S^3,\xi_{std})$.
\end{lemma}
\begin{proof}
It is clear that smoothly the branched cover of $S^3$ branched over the unknot is diffeomorphic to $S^3$. Moreover if $U$ is the transverse unknot in the standard contact structure on $S^3$ with self-linking number $-1$ then $S^3-U=S^1\times \R^2$ with the contact structure $\xi=\ker(d\phi + r^2\, d\theta)$, where $\phi$ is the angular coordinate on $S^1$ and $(r,\theta)$ are polar coordinates on $\R^2$. Now any finite covering space of $(S^1\times \R^2, \xi)$ is contactomorphic to $(S^1\times \R^2, \xi)$, from which the result follows.
\end{proof}

The hypothesis on the self-linking number in Lemma \ref{firstbc} is meaningful, since branched covers along stabilized transverse knots are overtwisted.

\subsection{Monodromy representations}
A covering map $p:M\to Y$, for $M$ connected, is determined by its monodromy representation. Let $x_0\in Y$ be a fixed base point and label the points in $p^{-1}(x_0)=\{x_1,\ldots, x_n\}$. Then, given a loop $\gamma$ in $Y$ based at $x_0$ we can define the element $\sigma_\gamma$ of the symmetric group $S_n$ by $\sigma_\gamma(i)$ being the index on $\widetilde\gamma(1)$ where $\widetilde\gamma$ is the lift of $\gamma$ based at $x_i$. This defines a transitive representation of the fundamental group
\[
r:  \pi_1(Y)\to S_n=\Aut(p^{-1}(x_0)).
\]
Moreover, given such a representation the subgroup $G=\{g\in \pi_1(Y): r(g)(1)=1\}$ gives a covering space $M$ that corresponds to the representation. 

Let $p:M\to Y$ be a branched covering map with branch locus $L\sse Y$. By definition, the restriction $p|_{(M-p^{-1}(L))}:M-p^{-1}(L)\to Y-L$ is a covering space and, by the above discussion, this restriction is determined by a representation $\pi_1(Y-L)\to S_n$, with $n=|p^{-1}(x_0)|$ the cardinality of the fiber over any $x_0\in Y-L$. In turn, such a representation determines the initial branched cover map $p:M\to Y$. Indeed, the representation of the fundamental group determines a cover of $Y-L$ and, in addition, any covering of $Y-L$ can be uniquely extended to a branched covering map of $Y$ by using the local model
$$L\times\C\to L\times\C,\quad(\theta,z)\mapsto (m\theta,z^n),$$
near the branch set, where the target $L\times\C$ is a sufficiently small neighborhood of $L\sse Y$ and $\C$ provides normal coordinates for this inclusion. The domain $L\times\C$ models part of a sufficiently small neighborhood of $p^{-1}(L)\sse M$. See \cite[Section 2.6]{EtnyreFurukawa17} and \cite{Geiges97} for further details.

In the case that $Y$ is smoothly $S^3$, the fundamental group $\pi_1(Y-L)$ is well-known to be generated by meridians to $L$. In particular, given a diagram for $L$ we have the Wirtinger presentation for $\pi_1(Y-L)$ with generators $x_i$ corresponding to the strands in the diagram and relations coming from the crossings: $x_i=x_k^{-1}x_jx_k$, respectively $x_i=x_kx_jx_k^{-1}$, at a right handed, respectively left handed, crossing, where strand $k$ goes over incoming strand $j$ and outgoing strand $i$. Thus a branched covering map over $S^3$ with branch set $L$ is determined by labeling a diagram of $L$ with elements of $S_n$ that satisfy the required relations at the crossings. 

It is easy to see whether a branched covering map is simple using the monodromy representation: one just needs the image of the representation to consist only of transpositions, as higher length permutations correspond to more than two points coming together.

We now discuss some modifications one can make to the branch locus of a branched covering map without affecting the induced contact covering space. 
\begin{lemma}
Let $L$ be a transverse link in a contact 3--manifold $(Y,\xi)$ and $p\colon M\to Y$ be a branched covering map with branch locus $L$ inducing the contact structure $\xi'$ on $M$. If part of a diagram for $L$ is as shown on one side of a row in Figure~\ref{isotopy} then replacing that portion of $L$ with the other diagram shown in that row will result in a new contact branched covering of $Y$ that still yields the same contact manifold $(M,\xi')$. 
\end{lemma}
\begin{figure}[htb]{\small
\begin{overpic}%[grid,tics=10] 
{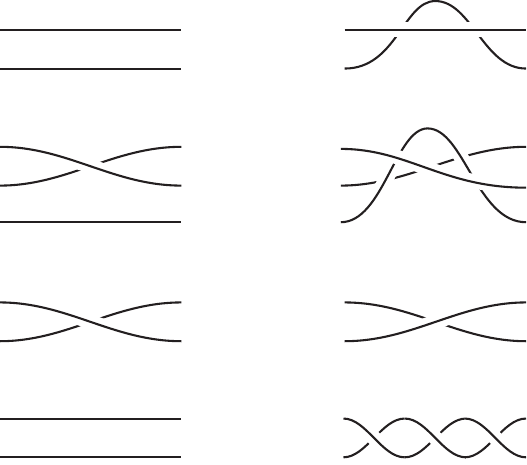}
% --------- Row 1 ---------
\put(-20, 187){$(i\, j)$}
\put(-20, 207){$(j\, k)$}%95+112
\put(90, 187){$(i\, j)$}
\put(90, 207){$(j\, k)$}
\put(145, 187){$(i\, j)$}
\put(145, 207){$(j\, k)$}
\put(201, 227){$(i\, k)$}
\put(255, 187){$(i\, j)$}
\put(255, 207){$(j\, k)$}

% --------- Row 2 ---------
\put(-20, 149){$(i\, j)$}
\put(-20, 130){$(i\, j)$}
\put(-20, 112){$(j\, k)$}
\put(90, 149){$(i\, j)$}
\put(90, 130){$(i\, j)$}
\put(90, 112){$(j\, k)$}
% ---------
\put(145, 147){$(i\, j)$}
\put(145, 130){$(i\, j)$}
\put(145, 112){$(j\, k)$}
\put(199, 165){$(i\, k)$}
\put(188, 128){$(i\, k)$}
\put(256, 112){$(j\, k)$}
\put(256, 149){$(i\, j)$}
\put(256, 130){$(i\, j)$}
\put(197, 148){$(j\, k)$}

% --------- Row 3 ---------
\put(-20, 57){$(i\, j)$}
\put(90, 75){$(i\, j)$}
\put(-20, 75){$(k\, l)$}
\put(90, 57){$(k\, l)$}
\put(147, 57){$(i\, j)$}
\put(256, 57){$(k\, l)$}
\put(147, 75){$(k\, l)$}
\put(256,75){$(i\,j)$}

% --------- Row 4 ---------
\put(-20,1){$(i\, j)$}
\put(-20, 19){$(j\, k)$}
\put(90,1){$(i\, j)$}
\put(90, 19){$(j\, k)$}
\put(146, 1){$(i\, j)$}%19+128
\put(146, 19){$(j\, k)$}
\put(188, 26){$(i\, k)$}
\put(256, 2){$(i\, j)$}
\put(256, 19){$(j\, k)$}
\end{overpic}}
\caption{Replacing a portion of the branch locus of a simple cover with one of the figures in a row with the other does not change the manifold or contact structure described by the contact branched cover. All strands are oriented from left to right.}
\label{isotopy}
\end{figure}
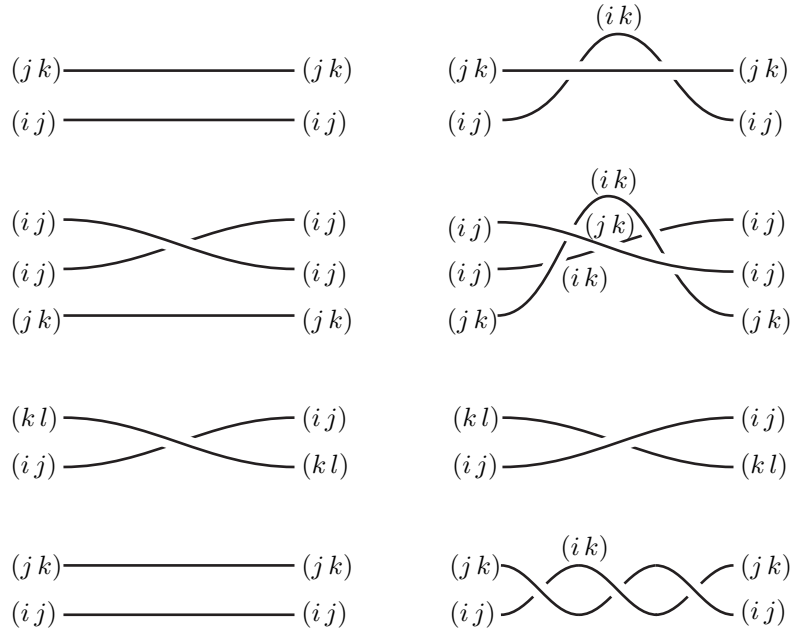
\begin{proof}
The changes to the diagram depicted in the first two rows of Figure~\ref{isotopy} correspond to contact isotopies of the underlying branch locus. This induces an isotopy of the branch cover map, thus does not change the branched cover or the contact structure. 

The change in the fourth row was proven to leave the contact branched cover unchanged in \cite{Casey13} and can easily be seen by observing that the branched cover of the ball containing either branched loci is simply a ball and the contact structure on it is tight. See \cite{Casey13} for details. 

That the change in the third row does not affect the contact branched cover follows exactly as for the fourth row. That is, the pre-image by the covering map of a ball containing the diagram on either side is a union of balls and the contact structure on each ball is tight. 
\end{proof}

There is a second simple modification to the transverse branched locus which preserves the contact isotopy type of the branched cover, it adds a standard transverse unknot to the branch locus at the cost of increasing the degree of the branched cover:

\begin{lemma}\label{add}
Let $L$ be a transverse link in a contact 3--manifold $(Y,\xi)$ and $p\colon M\to Y$ be an $n$--fold branched covering map with branch locus $L$ inducing the contact structure $\xi'$ on $M$. Let $L'$ be the link consisting of $L$ together with a maximal self-linking unknot $U$ that is contained in a ball that is disjoint from $L$.

The monodromy data for $L'$ given by using the monodromy data on $L$ and labelling the meridian of $U$ by $(n,n+1)$ defines an $(n+1)$--fold branch covering map from $M$ to $Y$ that induces the same contact structure $\xi'$ on $M$. 
\end{lemma}
\begin{proof}
We begin with a simple observation. Let $K=K_0\cup K_1$ be a link with monodromy data for an $n$--fold branched covering map, such that there is a sphere $S$ that separates $K_0$ and $K_1$. We can assume that $S$ is a convex sphere and cut $S^3$ into two contact 3--balls $B_0$ and $B_1$ along $S$ so that $B_i$ contains $K_i$. Then cap off $B_i$ by a tight contact 3--ball $C_i$ to get a tight contact 3-sphere $S^3_i$ that contains $K_i$. Let $p_i:M_i\to S^3_i$ be the $n$--fold branched covering map determined by $K_i$ and its monodromy data. Note that the pre-image $p_i^{-1}(C_i)$ consists of $n$ distinct 3--balls. Moreover if $p:Y\to S^3$ is the branched covering map corresponding to $K$ then $Y\setminus p^{-1}(S)$ is contactomorphic to $(Y_0\setminus p_0^{-1}(C_0))\cup (Y_1\setminus p_1^{-1}(C_1))$. In consequence, the contact manifold $Y$ is the contact $n$--fold, possibly internal, connect sum of $Y_0$ and $Y_1$. 

Now given the link $L'=L\cup U$ in the lemma notice that $Y_0=Y\cup S^3$ since $Y_0$ will be an $(n+1)$--fold cover of $S^3$ branched along $L$, where the monodromy data for $L$ only has labels between $1$ and $n$, so the cover is disconnected and the $(n+1)$--sheet is disjoint from the rest. Similarly $Y_1$ is the union of $n$ copies of 3--spheres by Lemma~\ref{firstbc}. It then follows that the $(n+1)$--fold contact connected sum of $Y_0$ and $Y_1$ results in a manifold contactomorphic to $Y$. 
\end{proof}

\section{Transverse universal links}

In this section we prove Theorem \ref{thm:univlink} by showing the link $L$ in Figure~\ref{tul} is a universal transverse link $L\sse (S^3,\xi_{std})$. The proof will follow from two following lemmas.
\begin{lemma}\label{setup}
Let $L'$ be the link shown in Figure~\ref{torusuniv}.
There is a branched cover
$$p:(S^3,\xi_{std}) \to (S^3,\xi_{std}),$$
branched along the Hopf link discussed at the end of Section~\ref{transknots} such that the link $L_{m,n}$ shown in Figure~\ref{norm} is a sub-link of $p^{-1}(L')$.
\end{lemma}
\begin{figure}[h]
\tiny
\begin{overpic}%[grid,tics=10] 
{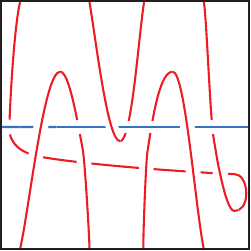}
%\put(-5,72){$r=$}
\end{overpic}
\caption{The front projection of the link $L'$ depicted in the contact manifold $(T^2\times (0,1),\xi_\std)\cong(S^3-H,\xi_\std)\sse(S^3,\xi_\std)$.}
\label{torusuniv}
\end{figure}
\begin{figure}[h]
\tiny
\begin{overpic}%[grid,tics=10] 
{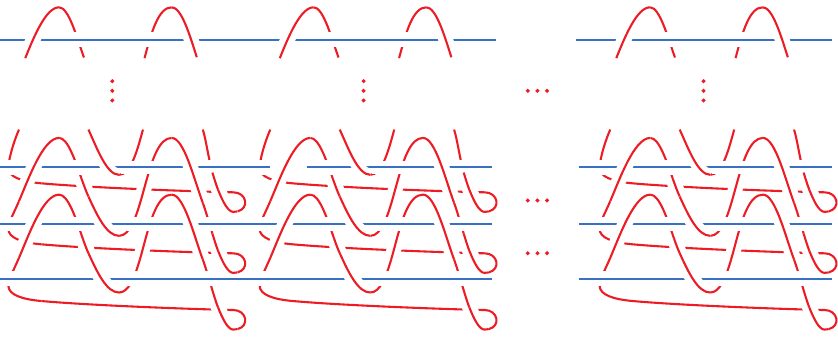}
%\put(-5,72){$r=$}
\end{overpic}
\caption{The transverse link $L_{m,n}$ in Lemma \ref{setup}. There are $n$ horizontal lines that are closed into a trivial braid and there are $m$ vertical columns of clasping unknots. Lemma \ref{semiuniversal} states that for any contact manifold $(Y,\xi)$ there is a choice of $n,m\in\N$ such that $(Y,\xi)$ is the contact branched cover of $(S^3,\xi_{std})$ over this transverse link.}
\label{norm}
\end{figure}  
\begin{lemma}\label{semiuniversal}
Any contact 3-manifold $(Y,\xi)$ can be realized as a contact branched cover of $(S^3,\xi_{std})$ branched along the transverse link $L_{n,m}$ shown in Figure~\ref{norm}, for some $n$ and $m$.
\end{lemma}

Lemmas \ref{setup} and \ref{semiuniversal}, which will be momentarily proven, suffice to conclude our main result.

\begin{proof}[Proof of Theorem~\ref{thm:univlink}]
Let $(M,\xi)$ be any given contact 3--manifold, we can apply Lemma~\ref{semiuniversal} to construct a branched covering map $p':(M,\xi)\to (S^3,\xi_{std})$ with branched set $L_{n,m}$. By Lemma~\ref{setup} there is also a branched covering map $p:(S^3,\xi_{std}) \to (S^3,\xi_{std})$ with branched set the green Hopf link $H$ depicted in Figure~\ref{tul} such that $L_{n,m}$ is a sub-link of $p^{-1}(L')$, where $L'\subseteq T^2\times (0,1)\simeq S^3\setminus H$ is the link presented in Figure~\ref{torusuniv}. Then Lemma~\ref{composition} implies that the composition $p'\circ p: (Y,\xi)\to (S^3,\xi_{std})$ is a contact branched covering map with branch set the transverse link $L'\subseteq (T^2\times (0,1),\xi_\std)\cong(S^3-H,\xi_\std)$ union the transverse Hopf link $H\subseteq (S^3,\xi_\std)$.

In order to conclude Theorem~\ref{thm:univlink} it suffices to identify the link in Figure~\ref{torusuniv} with the red and blue sub-link in Figure~\ref{tul}. This will be done in two steps. The first step is to show that the link in $(S^3,\xi_\std)$ shown in Figure~\ref{torusuniv} is the same as the link in $(\R^3, \xi_{rot}=\ker \{dz+xdy-ydx\})$ shown in Figure~\ref{fig:ULink}, where $(\R^3,\xi_{rot})$ is contactomorphic to the complement of a point in $(S^3,\xi_\std)$ disjoint from the link. Then in the second step we explain why the link in Figure~\ref{fig:ULink} is the link referred to in Theorem~\ref{thm:univlink}, depicted in Figure \ref{tul}, by identifying the contact structure $(\R^3,\xi_{std}=\ker\{dz-ydx\})$ with $(\R^3,\xi_{rot})$.
\begin{figure}[h]
\tiny
\begin{overpic}{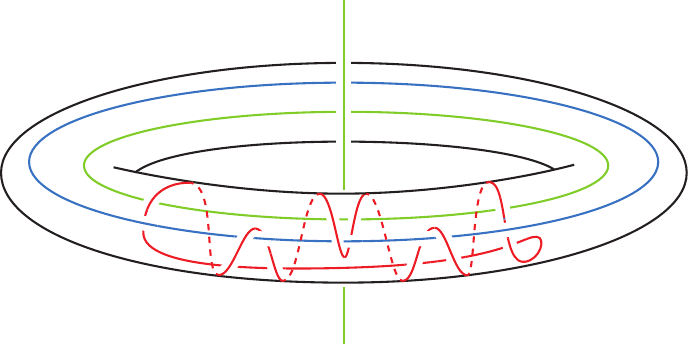}
\end{overpic}
\caption{The link $L'$ in $\xi_{rot}$}
\label{fig:ULink}
\end{figure}

For the first step, the stereographic coordinate map from $S^3-\{(0,0,1)\}$ to $\R^3$ pushes forward the contact structure $\xi_\std$ to the contact structure on $\R^3$ given by the kernel of the contact form $\alpha_1= r^2\, d\theta+ zr\, dr + \frac 12\left( 1+z^2-r^2\right)\, dz$,
where $(z,r,\theta)$ are cylindrical coordinates on $\R^3$. See the proof of \cite[Proposition 2.1.8]{Geiges08} for the detailed computation. The link in Figure~\ref{torusuniv} is mapped under the stereographic projection to the link in Figure~\ref{fig:ULink}. Given that the contact structure is $(\R^3,\ker\a_1)$, we need to apply a diffeomorphism to arrive at $(\R^3,\xi_{rot})$. It is shown in \cite[Page 57]{Geiges08} that the diffeomorphism 
\[
f:\R^3\longrightarrow\R^3,\quad f(r,\theta, z)=\left(r, \theta-z,\frac 12 z\left(1+\frac 13 z^2 + r^2\right)\right)
\] 
is a contactomorphism which pull-backs the contact structure $\xi_{rot}$ to $\ker \alpha_1$. Since the components of the link in Figure~\ref{fig:ULink}, with the exception of the vertical green line, can be assumed to have $r$-coordinate arbitrarily close to $1$ and $z$-coordinate arbitrarily close to $0$, we see that the transverse link in Figure~\ref{fig:ULink} represents the transverse link shown in Figure~\ref{torusuniv}. 

For the second step, we now need to identify the contact structures $(\R^3,\xi_{rot})$ and $(\R^3, \xi_{std})$. The contactomorphism from the first to the second is explicitly given by the diffeomorphism $\varphi(x,y,z)=(x,2y,xy+z)$.

It is a simple exercise, see for instance \cite[Figure 6]{KhandhawitNg10}, to show that the horizontal green and the blue curve in Figure~\ref{fig:ULink} map to the large green and blue curve in Figure~\ref{tul}. Similarly the vertical green curve, perturbed not to go through infinity, in Figure~\ref{fig:ULink} maps to the other green curve in Figure~\ref{tul}.

We must now see that the red curve in one figure maps to the red curve in the other. %, we consider the transverse link in Figure~\ref{fig:ULink}.
%in the solid torus, with $T^2\times(0,1)$ embedded in $(S^3,\xi_\std)$, see Section~\ref{transknots} for this embedding, which after stereographic projection becomes $(\R^3,\xi_{rot})$. 
For any given $\varepsilon\in\R^+$, the region containing the red component can be confined inside the region $|x|\leq \varepsilon$, and then we see that the contactomorphism $\varphi:(\R^3,\xi_{rot})\to(\R^3,\xi_{std})$ sends the red component of the transverse link in Figure \ref{fig:ULink} to the red component in the transverse link in Figure~\ref{tul}, as desired. 
\end{proof}

Let us now prove the two Lemmas \ref{setup} and \ref{semiuniversal}.

\begin{proof}[Proof of Lemma~\ref{setup}]
Let $p_{m,n}:(T^2\times (0,1))\to (T^2\times (0,1))$ be the covering map that unwinds the $\theta$-circle $m$ times and the $\phi$-circle $n$ times, where we are using the coordinates from Section~\ref{transknots}. This covering map smoothly extends to a branched covering map $p:S^3\to S^3$. Moreover, since $(p_{m,n})^*(\xi_{std})$ is contactomorphic to $\xi_{std}$, it is clear the contact structure induced on $S^3$ by this branched covering map is the standard contact structure. By construction, the link $L_{m,n}$ in Figure \ref{norm} is a sub-link of the pre-image $p^{-1}(L')$ of $L'$ via this explicit branched cover.
\end{proof}

\begin{proof}[Proof of Lemma~\ref{semiuniversal}]
Given a contact manifold $(M,\xi)$, we apply Theorem~\ref{giroux} to construct a 3--fold simple branched cover from $M$ onto $(S^3,\xi_{std})$ with branch locus a transverse braid $B$, such that $\xi$ is induced from the covering. Note that since the covering map is simple, the labels on the strands of $B$ are all transpositions.  We will now modify the transverse braid $B$ without changing the contact isotopy type of the branched cover $(M,\xi)$ until the $B$ has the desired form depicted in Figure \ref{norm}.
\begin{figure}[h]
\tiny
\begin{overpic}%[grid,tics=10] 
{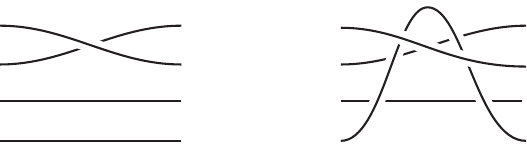}
\put(-16, 1){$(j\, k)$}
\put(-15, 21){$(i\, j$)}
\put(-15, 39){$(i\, j)$}
\put(-15, 57){$(i\, j)$}
\put(148, 1){$(j\, k)$}
\put(149, 21){$(i\, j$)}
\put(149, 39){$(i\, j)$}
\put(149, 57){$(i\, j)$}
\put(200, 14){$(i\, k)$}
\put(193, 38){$(i\, k)$}
\put(200, 71){$(i\, k)$}
\put(200, 54){$(j\, k)$}
\put(255, 1){$(j\, k)$}
\put(255, 21){$(i\, j)$}
\put(255, 39){$(i\, j)$}
\put(255, 57){$(i\, j)$}
\end{overpic}
\caption{Arranging all crossings to have multiple labelings on the strands.}
\label{multicolor}
\end{figure}
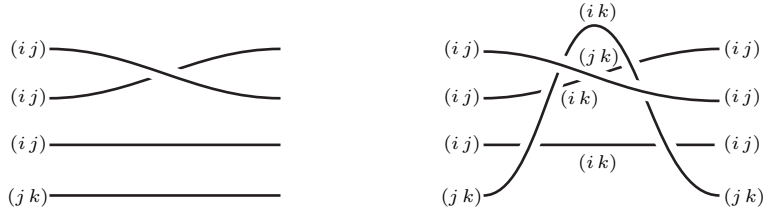  

\noindent
{\em Step 1: Change $B$ so that at each crossing the strands have three distinct labels.} The main idea in this step, described in detail below, is contained in Figure~\ref{multicolor}. Thinking of $B$ as the closure of a braid we assume $B$ has $n$ strands. We will think of a diagram for $B$ as $n$ horizontal strands the $xz$-plane with twists at distinct $x$ values, i.e.~ presented using standard generators of the braid group. Each of the $n$ different strands at any $x$ value must be labelled with at least two different transpositions in the symmetric group. If there is a single $x$ value where there is just one then all strands in the braid will be labelled by this transposition and the covering is not a 3--folding covering, but just a cyclic 2--fold covering.

\begin{figure}[h!]
\tiny
\begin{overpic}%[grid,tics=10] 
{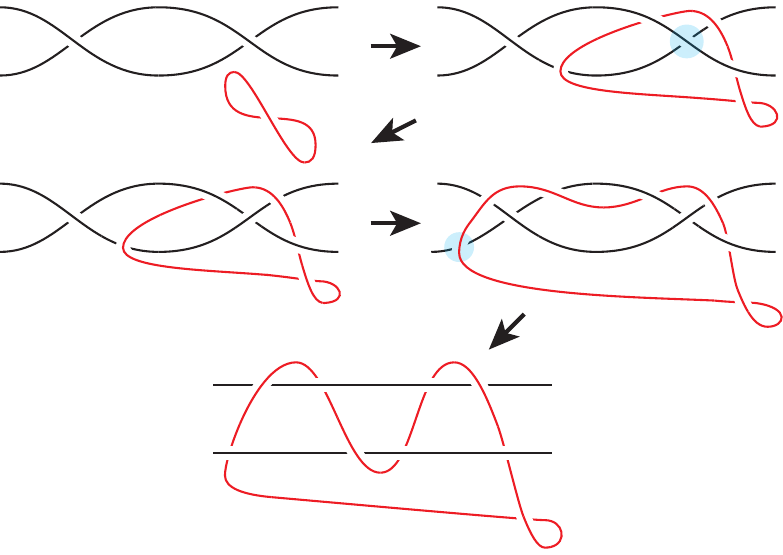}
\put(94, 218){$(j\, l)$}
\put(165, 229){$(i\, j)$}
\put(165, 263){$(i\, k)$}%242+21
\put(-15, 263){$(j\, k)$}
\put(-15, 229){$(i\, k)$}
\put(375, 229){$(i\, j)$}
\put(375, 263){$(i\, k)$}
\put(195, 263){$(j\, k)$}
\put(195, 229){$(i\, k)$}
\put(310, 214){$(j\, l)$}
\put(354, 243){$(i\, l)$}
\put(290, 236){$(k\, l)$}
\put(85, 130){$(j\, l)$}
\put(105, 146){$(k\, l)$}
\put(141, 163){$(i\, l)$}
\put(70, 183){$(i\, j)$}
\put(165, 145){$(i\, j)$}
\put(165, 179){$(i\, k)$}
\put(-15, 179){$(j\, k)$}
\put(-15, 145){$(i\, k)$}
\put(375, 145){$(i\, j)$}
\put(375, 179){$(i\, k)$}
\put(193, 178){$(j\, k)$}
\put(192, 145){$(i\, k)$}
\put(310, 115){$(j\, l)$}
\put(352, 161){$(i\, l)$}
\put(285, 138){$(k\, l)$}
\put(256, 163){$(i\, l)$}
\put(285, 182){$(i\, j)$}
\put(267, 47){$(i\, j)$}
\put(267, 80){$(i\, k)$}%242+21
\put(86, 80){$(j\, k)$}
\put(87, 47){$(i\, k)$}
\put(170, 16){$(j\, l)$}
\put(241, 65){$(i\, l)$}
\put(177, 86){$(k\, l)$}
\put(101, 65){$(j\, l)$}
\put(145, 65){$(j\, k)$}
\put(203, 65){$(i\, k)$}
\end{overpic}
\caption{Adding an unknotted component to the branch locus.}
\label{clasp}
\end{figure}
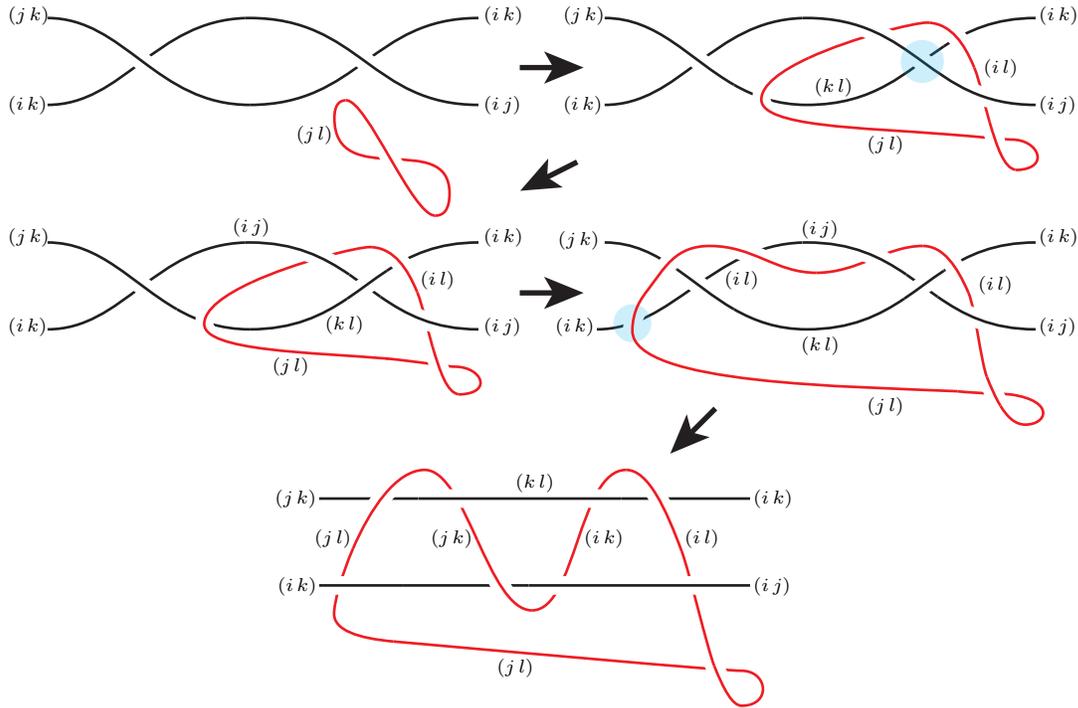

Now, let us suppose $c$ is a crossing in the transverse braid $B$. The strands at $c$ will either be labelled with three distinct transpositions or just one. In the former case, there is nothing to do, whereas in the latter case some other strand of the braid with the same $x$-value as the crossing must be labelled with a different transposition. Taking the strand with a different permutation that is closest to the crossing we can push it past the crossing as shown in Figure~\ref{multicolor}. This is done by a sequence of moves shown in the first two rows of Figure~\ref{isotopy} and results in a braid with distinct labels at each crossing. 

\noindent
{\em Step 2: Change $B$ so that it has only positive crossings and they occur in pairs.} We can apply the move shown in row four of Figure~\ref{isotopy} to each negative crossing in order to change it into a pair of two positive crossings. Note that we also need to use the move in row one to remove adjacent positive and negative crossings. For each unpaired positive crossing we can similarly add three more positive crossing and change it into two paired positive crossings. For convenience we can also isotope the diagram so that each pair of positive crossings lies in a distinct vertical strip of the projection.

\noindent
{\em Step 3: Change $B$ so that it is the trivial braid with linking unknots as in Figure~\ref{norm}.}
For each pair of positive crossings we can use Lemma~\ref{add} to introduce an unknot with label $(j\, l)$, where when considering this crossing we already have an $(l-1)$--fold branched cover and $j$ is chosen as shown in Figure~\ref{clasp}. Now Figure~\ref{clasp} shows a sequence of isotopies of the unknot and applications of the move in row three of Figure~\ref{isotopy} which manages to remove the pair of crossings at the expense of inserting a linked unknot. We will call such an unknot a {\it clasping unknot}.

If the braid after Step 2 had $k$ pairs of crossings, then after Step 3, there are $k$ clasping unknots and each one in a distinct vertical strip of the projection. Then adding $(n-2)$ new clasping unknots labelled with the identity permutation to the strands above and below each existing clasping unknot yields the link in Figure~\ref{norm}. Lemma \ref{semiuniversal} is now proven once we note that adding a knot to the branch locus whose strands are labelled with the identity permutation does not change the contact branched covering.
\end{proof}
\def\cprime{$'$} \def\cprime{$'$}

%\bibliography{references}
%\bibliographystyle{gtart}
%\bibliographystyle{plain}
\end{document}